\documentclass[a4paper]{amsart}
\usepackage{amssymb,mathtools,booktabs}
\usepackage[hidelinks]{hyperref}
\usepackage[margin=1in]{geometry}
\newtheorem{Theorem}{Theorem}[section]
\newtheorem{Lemma}[Theorem]{Lemma}
\newtheorem{Proposition}[Theorem]{Proposition}

\newtheorem{ExternalTheorem}[Theorem]{Theorem}
\numberwithin{equation}{section}
\newcommand{\E}{\mathbb E}
\newcommand{\Pp}{\mathbb P}
\newcommand{\ind}{\mathbf1}
\DeclareMathOperator{\Aut}{Aut}
\begin{document}

\title[Voronoi volume transforms in the Brownian map]{A marked scheme representation for Voronoi volumes in the Brownian map}
\author{Zhenyao Sun}
\address{School of Mathematics and Statistics, Beijing Institute of Technology, Beijing, China}
\email{zhenyao.sun@gmail.com}
\date{}
\begin{abstract}
For any fixed number $k\ge3$ of independent volume-sampled sites in the unit-area Brownian map, we represent the negative-power transform of their ordinary Voronoi volumes as a finite sum of marked planar scheme integrals. These integrals arise from the critical asymptotics of generating functions for face-constrained labelled planar maps. To connect them to Voronoi volumes, we use Miermont's zero-delay correspondence to interpret the labelled maps as quadrangulations with distinguished sites and then pass to the Brownian map scaling limit.
\end{abstract}
\maketitle
\section{Introduction}\label{sec:introduction}
\subsection{From nearest sites to a volume transform}
Sampling independent points from a probability measure gives a natural way to divide a metric space into regions: assign a point of the space to its nearest sampled site. The resulting Voronoi volumes sum to one whenever the set of distance ties has zero measure. Their joint distribution records the interaction of the metric and the measure. On a random surface, it describes how a typical collection of sites divides the total area.

The Brownian map is the random compact metric-measure space obtained as the scaling limit of uniform planar quadrangulations. A planar map is a finite connected graph embedded cellularly in the oriented sphere, considered up to orientation-preserving homeomorphism; its faces are the complementary disks, and loops and multiple edges are allowed. A quadrangulation is a planar map whose faces all have degree four. Graph distances are rescaled by a constant times the inverse fourth root of the number of faces. The convergence theorem of Miermont~\cite{miermont} identifies the limiting distance using the labelled-tree encoding, and uniform vertex measure converges to the volume measure of the limit. This discrete approximation suggests studying Brownian map Voronoi volumes by counting quadrangulations with several distinguished vertices. There is, however, no immediate formula for these volumes from the tree encoding.

The Brownian map Voronoi cell-volume problem has already yielded explicit moment formulas. Let $V_i^{(k)}$ denote the volume of the $i$th Voronoi cell generated by $k$ independent volume-sampled sites. In particular, Chapuy~\cite[Theorems~2--3]{chapuy2019} established that
\[
 \E\bigl[(V_1^{(2)})^2\bigr]=\frac13,
 \qquad \E\bigl[V_1^{(3)}V_2^{(3)}V_3^{(3)}\bigr]=\frac1{60}.
\]
These identities give the second moment of a cell in the two-site case and a mixed third moment in the three-site case. They coincide with the corresponding moments of the spacings made by $k-1$ independent uniform points in a unit interval. This agreement motivates Chapuy's conjecture that the entire Voronoi volume vector has the Dirichlet distribution with all parameters equal to one, namely the uniform probability distribution on the simplex of nonnegative coordinates summing to one. In symbols, the proposed law is
\[
 (V_1^{(k)},\ldots,V_k^{(k)})\sim\operatorname{Dirichlet}(1,\ldots,1).
\]

Guitter~\cite{guitter2017} approached the two-site problem through generating functions for planar quadrangulations. The probabilistic question is the division of area between two uniformly sampled vertices in a large quadrangulation. To count that area, his approach uses Miermont's multisource correspondence \cite{miermont2009}, which encodes quadrangulations with distinguished vertices by labelled maps. In its zero-delay form, and under the source conditions recalled in Section~\ref{sec:faces}, a quadrangulation with $k$ ordered sources corresponds to a labelled planar map with $k$ numbered faces, one for each source. These are two different maps: the quadrangulation has $n$ faces of degree four, whereas the labelled map has $n$ edges and $k$ faces whose degrees may vary. The vertex labels record distances to the nearest source, and half the degree of each labelled-map face gives a discrete area observable associated with the corresponding Voronoi cell.

For two sources, Guitter's generating function $F(g,h)$ assigns weights $g$ and $h$ to these two half-face degrees. It is expressed as a sum whose $s$th term counts sources at graph distance $2s$, and thus already sums over all admissible even distances \cite[Section~2, equation~(1)]{guitter2017}. The critical asymptotics of this function lead to a uniform limiting area fraction, or the $\operatorname{Dirichlet}(1,1)$ law for the two fractions \cite[Sections~3--4]{guitter2017}. The passage from these asymptotics to fixed-size coefficients involves heuristic steps; Chapuy~\cite[p.~481, footnote~2]{chapuy2019} describes the calculation as semi-rigorous. It nevertheless suggests retaining a separate area variable for each site when studying more than two cells.

For a fixed $k\ge3$, we therefore consider labelled planar maps $M$ with $k$ numbered faces. Write $\ell$ for the integer vertex labels. Adjacent labels differ by at most one, and the minimum boundary label of each face is exactly one: if $\partial i$ denotes the boundary vertices of face $i$, then
\[
 \begin{gathered}
 \ell:V(M)\to\mathbb Z,\qquad
 |\ell(v)-\ell(w)|\le1\quad(vw\in E(M)),\\
 \min_{v\in\partial i}\ell(v)=1\quad(1\le i\le k).
 \end{gathered}
\]
For such a map $M$, let
\[
 a_i(M)=\frac12\deg_M(i)
\]
be half the boundary degree of face $i$. Boundary sides are counted with multiplicity: an edge incident to the same face on both sides contributes twice to that face's degree. Consequently,
\[
 \sum_{i=1}^k a_i(M)=|E(M)|.
\]
Two labelled maps represent the same object if an orientation-preserving isomorphism preserves their vertex labels and face numbers. Thus different drawings of the same embedded map are counted only once. We write $[M]$ for an isomorphism class and $\Aut(M)$ for the finite group of its symmetries preserving this information. The factor $1/|\Aut(M)|$ gives each class the reciprocal of its number of symmetries.

Let the positive variable $g_i$ record the area $a_i(M)$. Summing these area weights, with the symmetry factors just described, defines
\begin{equation}\label{eq:face-series-definition}
 \mathcal Z_k(g_1,\ldots,g_k)=
 \sum_{[M]}\frac1{|\Aut(M)|}\prod_{i=1}^k g_i^{a_i(M)}.
\end{equation}
The sum runs over all maps in this class and all their sizes. Since every summand is nonnegative, it is understood as the increasing limit of the sums over maps with at most $N$ edges, initially allowing $+\infty$.

When all variables equal $g$, the identity for the total area gives the ordinary size-generating series
\[
 \mathcal Z_k(g,\ldots,g)=\sum_{n\ge0}A_{n,k}g^n,
 \qquad
 A_{n,k}=\sum_{\substack{[M]\\ |E(M)|=n}}\frac1{|\Aut(M)|}.
\]
Thus $A_{n,k}$ is the sum of the reciprocal symmetry orders of the size-$n$ maps. The counting argument in Lemma~\ref{ms:lem:v8-counting-normalization} gives
\[
 A_{n,k}\sim\frac{2^{-k}}{\sqrt\pi}\,12^n n^{k-7/2}\qquad(n\to\infty).
\]
In particular, $\lim_{n\to\infty}A_{n,k}^{1/n}=12$, so the radius of convergence is $g_c=1/12$. The power $k-7/2>-1$ also shows that the series diverges at that radius. Approaching $g_c$ from below therefore probes large maps: the contribution of any fixed finite set of sizes becomes negligible relative to the full series.

We let the face variables approach this common critical value at possibly different positive rates. Following the normalization of Sun~\cite{sun2026critical}, put
\begin{equation}\label{eq:intro-scaling-constants}
 g_i(\varepsilon)=\frac{1-\lambda_i\varepsilon^4/36}{12},
 \qquad \lambda_i>0,\quad \varepsilon\downarrow0.
\end{equation}
Here $1-12g_i(\varepsilon)=\lambda_i\varepsilon^4/36$ is the relative deviation from the critical value; the fourth power and the factor $36$ fix the scaling convention. For sufficiently small $\varepsilon$, every $g_i(\varepsilon)$ lies in $(0,1/12)$, and the nonnegative series is finite by comparison with the equal-variable series. The critical asymptotic result of Sun~\cite[Theorem~1.1]{sun2026critical} identifies the rescaled limit
\[
 \lim_{\varepsilon\downarrow0}\varepsilon^{4k-10}
 \mathcal Z_k(g_1(\varepsilon),\ldots,g_k(\varepsilon))
\]
as a finite sum of marked planar scheme integrals. The goal of this paper is to identify these integrals with the negative-power transform of Brownian map Voronoi volumes, up to an explicit normalization constant. The analytic critical asymptotics are supplied by that earlier work; the probabilistic identification is proved here. We now define the transform and the scheme integrals before stating the representation.

\subsection{The volume transform and the main result}\label{subsec:definitions}
We study ordinary distance Voronoi cells of independent volume samples in the unit-area Brownian map. The probability statements below average over both the random space and the sampled points.

Let $(M,d,\mu)$ denote the Brownian map with $\mu(M)=1$. Conditional on this space, let $X_1,\ldots,X_k$ be independent with law $\mu$, and set
\[
 C_i^{(k)}=\{z\in M:d(z,X_i)<d(z,X_j)\text{ for every }j\ne i\},
 \qquad V_i^{(k)}=\mu(C_i^{(k)}).
\]
For a finite label set $S$ of size $k$, we use the same definitions with the sites indexed by $S$ and abbreviate $V_i^{(k)}$ to $V_i$.

For a finite label set $S$ of size $k\ge3$, put
\begin{equation}\label{eq:volume-transform-definition}
 F_S(\lambda)=\E\left[\left(\sum_{i\in S}\lambda_iV_i\right)^{5/2-k}\right],
 \qquad \lambda\in(0,\infty)^S.
\end{equation}
Initially the expectation is interpreted in $[0,+\infty]$, with a negative power of zero equal to $+\infty$. The proof below establishes that the strict cells have total volume one, and hence that this transform is finite.

We next recall the marked scheme integrals of Sun~\cite[Section~1.2]{sun2026critical}. Their ingredients describe the critical limit of the labelled-map decomposition: trees attached to a fixed finite map give a scaling profile, and chains of degree-two vertices give Green kernels. We define these analytic ingredients first, then the finite maps and their face marks.

For $\eta>0$ let $a=(4\eta)^{1/4}$ and define
\begin{equation}\label{eq:potential}
 u_\eta(x)=\frac{a^2}{2}\left(1+3\operatorname{csch}^2(ax/2)\right),\qquad x>0.
\end{equation}
This profile solves $u_\eta''=u_\eta^2-\eta$, $u_\eta(\infty)=\sqrt\eta$, and
\begin{equation}\label{eq:potential-expansion}
 u_\eta(x)=\frac6{x^2}+\frac{\eta x^2}{10}-\frac{\eta^{3/2}x^4}{126}+O(x^6),\qquad x\downarrow0.
\end{equation}
For a positive parameter vector $\eta$, write $u_i=u_{\eta_i}$ and $p_i=-u_i'>0$. Let $K_{ij}$ be the positive Green kernel of the Friedrichs operator
\[
 H_{ij}=-\frac{d^2}{dx^2}+u_i+u_j
\]
on $(0,\infty)$. Here the Friedrichs realization is obtained by closing the positive quadratic form
\[
 f\longmapsto\int_0^\infty \bigl(|f'(x)|^2+(u_i(x)+u_j(x))|f(x)|^2\bigr)\,dx,
 \qquad f\in C_c^\infty(0,\infty).
\]
It selects the regular behavior at zero for this singular potential. Thus $K_{ij}$ is the integral kernel of $H_{ij}^{-1}$. In particular, its derivative in the first variable has jump $-1$ at the second variable. The regular solution at zero and the decaying solution at infinity of $H_{ij}f=0$ can be normalized by
\[
 \phi_{ij}(x)\sim x^4/7,\qquad \psi_{ij}(x)\sim x^{-3}\quad(x\downarrow0),
 \qquad \psi_{ij}(\infty)=0.
\]
With these normalizations $\phi_{ij}'\psi_{ij}-\phi_{ij}\psi_{ij}'=1$ and
\[
 K_{ij}(x,y)=\phi_{ij}(\min(x,y))\psi_{ij}(\max(x,y)).
\]

Recall that the cyclic order of incidences at each vertex is part of a planar map. For a finite label set $S$ with $|S|=k\ge3$, a \emph{cubic scheme with faces labelled by $S$} is such a planar map with every vertex of degree three and with its faces bijectively labelled by $S$. Euler's formula and the degree-three condition give $|V(G)|=2k-4$ and $|E(G)|=3k-6$. In particular, for fixed $k$ only finitely many such schemes occur.

Fix such a scheme $G$, and write $V(G)$ and $E(G)$ for its vertex and edge sets. Define the finite set
\[
 \mathcal A_G=\{\alpha:S\longrightarrow E(G):
 \alpha(i)\text{ is incident to face }i\text{ for every }i\in S\}.
\]
Thus each face chooses an incident edge. An edge incident to the same face on both sides occurs only once among that face's possible choices. The map $\alpha$ need not be injective: two distinct faces can choose their common edge.

Let $x=(x_v)_{v\in V(G)}\in(0,\infty)^{V(G)}$ and $z=(z_i)_{i\in S}\in(0,\infty)^S$. For an edge $e$, write $v,w$ for its endpoint incidences and $i,j$ for its incident face labels, allowing $v=w$ and $i=j$. Fix $\alpha\in\mathcal A_G$ and put $S_e(\alpha)=\alpha^{-1}(\{e\})\subseteq\{i,j\}$. We define the edge factor $F_\alpha(e;x,z)$ by the following four cases: an unmarked edge, one mark on an edge separating different faces, one mark on an edge with the same face on both sides, and two marks on a shared edge. If $S_e(\alpha)=\varnothing$, put
\[
 F_\alpha(e;x,z)=K_{ij}(x_v,x_w).
\]
If $i\ne j$ and $S_e(\alpha)=\{h\}$, where $h\in\{i,j\}$, put
\[
 F_\alpha(e;x,z)=K_{ij}(x_v,z_h)p_h(z_h)K_{ij}(z_h,x_w).
\]
If $i=j$ and $\alpha(i)=e$, put
\[
 F_\alpha(e;x,z)=2K_{ii}(x_v,z_i)p_i(z_i)K_{ii}(z_i,x_w).
\]
If $i\ne j$ and $\alpha(i)=\alpha(j)=e$, put
\begin{align*}
 F_\alpha(e;x,z)={}&K_{ij}(x_v,z_i)p_i(z_i)K_{ij}(z_i,z_j)p_j(z_j)K_{ij}(z_j,x_w)\\
 &+K_{ij}(x_v,z_j)p_j(z_j)K_{ij}(z_j,z_i)p_i(z_i)K_{ij}(z_i,x_w).
\end{align*}
An arbitrary ordering of the two endpoint incidences of $e$ may be used: symmetry of $K_{ij}$ preserves the first three cases and interchanges the two summands in the fourth. Both summands remain for a loop, even if their values coincide. Distinct parallel edges remain distinct elements of $E(G)$.

Let $\mathcal G_S$ be the set of isomorphism classes of cubic schemes with faces labelled by $S$, and let $\Aut(G)$ be the group of orientation-preserving, face-label-preserving automorphisms. We define
\begin{equation}\label{eq:scheme-sum-definition}
 \begin{split}
 \mathcal I_S(\eta)={}&\sum_{[G]\in\mathcal G_S}\frac1{|\Aut(G)|}
 \sum_{\alpha\in\mathcal A_G}
 \int_{(0,\infty)^{V(G)}}\int_{(0,\infty)^S}\\
 &\qquad\prod_{e\in E(G)}F_\alpha(e;x,z)\,
 \prod_{i\in S}dz_i\prod_{v\in V(G)}dx_v.
 \end{split}
\end{equation}
We allow $+\infty$ until finiteness is proved.

\begin{Theorem}[The volume-transform representation]\label{ms:lem:general-bridge}
For every finite label set $S$ with $|S|=k\ge3$ and positive parameter vector $\lambda=(\lambda_i)_{i\in S}$, put
\[
 \alpha_k=k-\frac52,\qquad
 C_k=\frac{2^{-k}36^{\alpha_k}\Gamma(\alpha_k)}{\sqrt\pi}.
\]
Then the integral is finite and
\begin{equation}\label{eq:bridge}
 3^{3k-6}\mathcal I_S(\lambda/4)=C_kF_S(\lambda),
\end{equation}
where $\lambda/4=(\lambda_i/4)_{i\in S}$. Both sides of \eqref{eq:bridge} depend real-analytically on the positive parameters $\lambda_i$. In particular,
\begin{equation}\label{eq:diagonal}
 \mathcal I_S((t)_{i\in S})
 =\frac{\Gamma(k-5/2)}{2^k3^{k-1}\sqrt\pi}\,t^{5/2-k},\qquad t>0.
\end{equation}
\end{Theorem}

Thus the negative-power volume transform is a finite, explicitly normalized sum of positive integrals. The identity holds for every positive choice of the parameters, and both expressions are real analytic on that entire domain. The last formula evaluates the integral when all parameters coincide; it does not determine its values off the diagonal.

In particular, the representation does not establish Chapuy's proposed Dirichlet law.

\subsection{Proof overview}
The analytic and probabilistic descriptions concern the same rescaled generating series. With the parameters in \eqref{eq:intro-scaling-constants} and the constants in Theorem~\ref{ms:lem:general-bridge}, we have
\begin{equation}\label{eq:two-limits}
\begin{aligned}
 \varepsilon^{4k-10}\mathcal Z_k(g(\varepsilon))
 &\longrightarrow 3^{3k-6}\mathcal I_{\{1,\ldots,k\}}(\lambda/4)
 &&\text{(analytic)},\\
 \varepsilon^{4k-10}\mathcal Z_k(g(\varepsilon))
 &\longrightarrow C_k\E\left(\sum_{i=1}^k\lambda_iV_i^{(k)}\right)^{-\alpha_k}
 &&\text{(probabilistic)}.
\end{aligned}
\end{equation}
Here $g(\varepsilon)=(g_1(\varepsilon),\ldots,g_k(\varepsilon))$, and $\mathcal I_S$ is the marked scheme integral just defined. The first convergence is the critical asymptotic result of Sun~\cite[Theorem~1.1]{sun2026critical}. We prove the second by relating the generating series to Brownian map Voronoi volumes. Comparing the two limits gives the representation.

The probabilistic identification relies on two geometric ingredients. The first concerns the bipartite structure of quadrangulations. Miermont's zero-delay correspondence requires the distinguished vertices to belong to the same bipartite color class, whereas the Voronoi sites in the Brownian map are sampled independently from its volume measure. To reconcile these settings, we prove that the relative colors of uniformly sampled vertices converge to independent fair bits, independently of the limiting marked Brownian map. Thus the color restriction does not alter the limiting marked geometry. Parity constraints also occur in Chapuy's enumeration \cite[Lemma~11 and the proof of Lemma~6]{chapuy2019}; the joint independence from the geometry is what is needed here.

The second ingredient concerns the relation between face degrees and Voronoi volumes. Under Miermont's correspondence, half the degree of a labelled-map face records a discrete area associated with the corresponding source. This area need not agree exactly with the size of the strict Voronoi cell, because vertices equidistant from two or more sources may contribute differently. We show that the discrepancies vanish after normalization, using the fact that distance ties have zero volume in the Brownian map. Together with the exact source and rooting counts, these two ingredients yield the probabilistic limit in \eqref{eq:two-limits}.

The proof occupies the remainder of the paper. Section~\ref{sec:ties} establishes that distance ties have zero volume in the Brownian map and proves continuity of Voronoi cell volumes under marked Gromov--Hausdorff--Prokhorov convergence. Section~\ref{sec:parity} shows that the bipartite color restriction on the distinguished vertices required by Miermont's zero-delay correspondence does not affect the limiting marked geometry. Section~\ref{sec:faces} establishes convergence of the normalized face half-degrees to Brownian map Voronoi volumes and determines the corresponding enumeration constant. Section~\ref{sec:comparison} compares the probabilistic and analytic limits and proves Theorem~\ref{ms:lem:general-bridge}.

\section{Null bisectors and continuity of cell masses}\label{sec:ties}
The probabilistic identification requires strict Voronoi cells to exhaust the volume and discrete tie masses to disappear. We obtain the necessary null bisectors from the atomlessness of the distance from a volume-sampled center; the same input gives null sphere boundaries.

The distance-profile result below is the atomlessness consequence of the continuous density profile in \cite[Proposition~3.7, p.~432]{legall2014}.
\begin{ExternalTheorem}[Atomless distance profile]\label{ext:distance-profile}
If $Y$ is sampled from $\mu$, then almost surely the measure
$\nu_Y(A)=\int_M\ind_{\{d(Y,z)\in A\}}\,\mu(dz)$ on $[0,\infty)$ has no atoms.
\end{ExternalTheorem}
\begin{Lemma}[Null distance ties]\label{ms:lem:bisectors}
Almost surely, $\mu\{z:d(z,X_i)=d(z,X_j)\}=0$ whenever $i\ne j$. The strict Voronoi cells therefore partition $M$ up to a null set.
For every finite collection of independent volume-sampled centers, all spheres centered at any of them have zero volume almost surely.
\end{Lemma}
\begin{proof}
Apply Theorem~\ref{ext:distance-profile}, and sample an additional point $Y$ independently from $\mu$, conditional on $(M,d,\mu)$. Conditional on $(M,d,\mu,Y)$, the distances $d(Y,X_i)$ and $d(Y,X_j)$ are independent with common atomless distribution $\nu_Y$. Consequently,
\[
 \Pp\bigl(d(Y,X_i)=d(Y,X_j)\mid M,d,\mu,Y\bigr)=0.
\]
Tonelli's theorem gives
\[
\begin{aligned}
 &\Pp\bigl(d(Y,X_i)=d(Y,X_j)\bigr)\\
 &\qquad=\E\left[\mu\bigl(\{z\in M:d(z,X_i)=d(z,X_j)\}\bigr)\right].
\end{aligned}
\]
The volume on the right is nonnegative and therefore vanishes almost surely. A finite union over pairs proves the partition statement. By Theorem~\ref{ext:distance-profile}, every sphere centered at any one independent volume-sampled point has zero volume almost surely. A finite union of the exceptional events proves this simultaneously for any finite collection of sampled centers.
\end{proof}

We will use convergence of compact metric spaces with probability measures and finitely many ordered marks in the marked Gromov--Hausdorff--Prokhorov topology. One characterization uses correspondences between the spaces whose distortions tend to zero, which contain all marked pairs, and couplings of the measures whose mass outside those correspondences tends to zero. Distortion measures the largest discrepancy between distances of two paired points. The following deterministic fact isolates the continuity needed for Voronoi cells: away from ties, a small perturbation of distances cannot change the nearest site.
\begin{Lemma}[Continuity of Voronoi masses]\label{ms:lem:cell-mass-continuity}
Fix $k\ge1$. Let
\[
(M_n,d_n,\mu_n;x_{n,1},\ldots,x_{n,k})
\quad\text{and}\quad
(M,d,\mu;x_1,\ldots,x_k)
\]
be compact metric spaces equipped with Borel probability measures and $k$ ordered marks. Suppose the marked spaces converge in the Gromov--Hausdorff--Prokhorov topology and
\[
\mu\{z\in M:d(z,x_i)=d(z,x_j)\}=0\qquad(1\le i<j\le k).
\]
Define the strict cells $C_{n,i}=\{z\in M_n:d_n(z,x_{n,i})<d_n(z,x_{n,j})\text{ for every }j\ne i\}$ and $C_i$ by the corresponding inequalities on $M$. Then
\[
(\mu_n(C_{n,1}),\ldots,\mu_n(C_{n,k}))
\longrightarrow(\mu(C_1),\ldots,\mu(C_k)).
\]
Moreover, the $\mu_n$-mass of points having more than one nearest mark tends to zero. Any measurable assignment of these tied points to their minimizing marks therefore gives the same limiting vector.
\end{Lemma}
\begin{proof}
For $k=1$ the unique cell is the whole space and there are no ties. Assume $k\ge2$. Realize marked Gromov--Hausdorff--Prokhorov convergence by correspondences $\mathcal R_n\subset M_n\times M$ containing each pair $(x_{n,i},x_i)$, with distortions $\varepsilon_n\to0$, and couplings $\pi_n$ of $\mu_n,\mu$ for which $e_n=\pi_n(\mathcal R_n^c)\to0$. Thus for $(z_n,z)\in\mathcal R_n$,
\[
|d_n(z_n,x_{n,i})-d(z,x_i)|\le\varepsilon_n
\quad(1\le i\le k).
\]
For $\delta>0$ put
\[
T_\delta=\bigcup_{i<j}\{z\in M:|d(z,x_i)-d(z,x_j)|\le\delta\}.
\]
As $\delta\downarrow0$ these sets decrease to the finite union of the pairwise tie sets, so $\mu(T_\delta)\to0$. If $2\varepsilon_n<\delta$, then for every $(z_n,z)\in\mathcal R_n$ with $z\notin T_\delta$ all pairwise distance comparisons have the same strict sign in the two spaces. In particular both points have the same unique nearest mark. Integration against $\pi_n$ gives simultaneously for all $i$
\[
|\mu_n(C_{n,i})-\mu(C_i)|\le e_n+\mu(T_\delta).
\]
The set of $z_n$ with more than one nearest mark has mass at most the same bound, because it cannot occur on such a good coupled pair. First let $n\to\infty$, then let $\delta\downarrow0$. This proves both limits. Changing an assignment of tied points changes any one cell mass by at most their total mass, which proves the final assertion.
\end{proof}

\section{Bipartite colors and the limiting geometry}\label{sec:parity}
The zero-delay correspondence requires all sources to have the same bipartite color, whereas our limiting sites are independent samples from Brownian-map volume. We must therefore prove that imposing the color condition preserves the marked metric-measure limit. A uniform limiting vector of relative color bits would determine the probability of that condition but would not show its independence from the geometry.

Our mechanism uses the pointed tree encoding. Each sampled vertex has, with high probability, a terminal part of its root path used by none of the other sampled vertices. Resampling a logarithmic number of increments on each such private path randomizes the label parities independently. The change in rescaled labels vanishes, and the theorem identifying the limiting distance from those labels shows that the marked geometry is asymptotically unchanged. We first establish this geometric stability, then deduce parity independence and the conditional limit.

\subsection{The pointed tree encoding}
To remove the color restriction without changing the limiting geometry, we use one common tree encoding for the metric and the label parities. The two source results below fix the same encoding, its normalization, and the identification of every subsequential distance limit.

Let $e$ be a normalized Brownian excursion on $[0,1]$. Conditionally on $e$, let $Z$ be the continuous centered Gaussian process with
\[
\E[Z_sZ_t\mid e]=\min_{s\wedge t\le u\le s\vee t}e_u.
\]
The tree distance encoded by $e$ is
\[
d_e(s,t)=e_s+e_t-2\min_{s\wedge t\le u\le s\vee t}e_u.
\]
For $s\le t$ set
\[
D^\circ(s,t)=Z_s+Z_t-2\max\left\{\min_{u\in[s,t]}Z_u,
\min_{u\in[0,s]\cup[t,1]}Z_u\right\},
\]
and extend symmetrically. For $s,t\in[0,1]$, define
\[
D^*(s,t)=\inf\left\{\sum_{h=1}^m D^\circ(a_h,b_h):\begin{array}{l}
m\in\mathbb N,\quad m\ge1,\quad a_1=s,\quad b_m=t,\\
a_h,b_h\in[0,1]\quad(1\le h\le m),\\
d_e(b_h,a_{h+1})=0\quad(1\le h<m)
\end{array}\right\}.
\]
Thus one pays $D^\circ$ for each link and may move freely between times representing the same tree vertex. The quotient by $D^*=0$, with the induced distance and the pushforward of Lebesgue measure, is the unit-area Brownian map in the normalization used here.

For $n\ge1$, let $T_n$ be a uniform rooted plane tree with $n$ edges. Give its root label zero and, conditionally on the tree, give its edges independent uniform label increments in $\{-1,0,1\}$, measured from parent to child. Write $L(v)$ for the resulting vertex labels, and let $\epsilon_n$ be an independent uniform sign in $\{-1,1\}$, called the root-orientation bit.

We use the following deterministic pointed Cori--Vauquelin--Schaeffer coding. List the tree corners in contour order, beginning at the root corner. The successor of a corner is the first subsequent corner in cyclic order with label one less; if there is none, its successor is an extra vertex $v_*$. Join each corner to its successor by an arc, drawn without crossings in the complement of the tree, and then erase the tree edges. Root the resulting map at the arc from the root corner to its successor: orient it toward the root corner when $\epsilon_n=1$ and away from that corner when $\epsilon_n=-1$. Thus the bit changes only the orientation of the root edge.

The coding properties below are given in \cite[Sections 4.1 and 5.4.2--5.4.3, Theorem 5.5, Corollaries 5.6--5.7, and Proposition 5.8]{legallmiermont2012}.
\begin{ExternalTheorem}[Pointed tree coding]\label{ext:pointed-tree-coding}
For every rooted plane tree $T$ with $n\ge1$ edges, integer-valued labels $L$ that are zero at the root and change by at most one along each edge, and sign $\epsilon\in\{-1,1\}$, the preceding construction gives a well-defined rooted quadrangulation $Q$ with $n$ faces and vertex set $V(T)\cup\{v_*\}$. For each $n\ge1$, this construction is a bijection from these labelled rooted plane trees together with a sign onto rooted quadrangulations with $n$ faces together with a distinguished vertex $v_*$. Both sides are taken up to orientation-preserving isomorphism respecting the root and the specified labels or distinguished vertex. For every $v\in V(T)$,
\[
d_Q(v,v_*)=L(v)-\min_{w\in V(T)}L(w)+1.
\]
For the random inputs $(T_n,L,\epsilon_n)$ specified above, $Q$ is uniform among rooted quadrangulations with $n$ faces, and $v_*$ is conditionally uniform among its $n+2$ vertices.
\end{ExternalTheorem}
Fix this coding. If $v_j$ is the tree vertex visited at contour time $j$, $0\le j\le2n$, let $C_n(j)$ be its height and $L_n(j)$ its label, and let $d_n(i,j)=d_Q(v_i,v_j)$. Interpolate these functions linearly in one variable and bilinearly for $d_n$.

The following subsequential convergence and limit identification are given in \cite[Section 2.1 and Theorem 3]{miermont}.
\begin{ExternalTheorem}[Subsequential encoding and metric limits]\label{ext:encoding-metric-convergence}
With $c_n=(9/(8n))^{1/4}$, from every subsequence of the random triples
\[
\left(\frac{C_n(2n\,\cdot)}{\sqrt{2n}},\ c_nL_n(2n\,\cdot),\ c_n d_n(2n\,\cdot,2n\,\cdot)\right)
\]
one can extract a further subsequence converging jointly in distribution, for the uniform norms on $[0,1]$ and $[0,1]^2$, to $(e,Z,D)$, where $(e,Z)$ has the law specified above. For every such subsequential limit, $D=D^*$ almost surely, with $D^*$ constructed from this same pair $(e,Z)$ by the formula above.
\end{ExternalTheorem}

The last theorem identifies the limiting distance as a measurable function of the same limiting contour and label processes. This is stronger than an unmarked convergence statement and is what makes the resampling argument below work. Two encodings with asymptotically identical contours and labels must have identical distance fields in every joint subsequential limit.

\subsection{Resampling private paths}
A label parity is the sum modulo two of the increments along a root path. We can nearly randomize several such parities independently by replacing increments on parts of these paths used by only one sampled vertex. The number of replaced edges tends to infinity, but remains negligible on both the tree-height and label scales.
\begin{Lemma}[Resampling preserves limiting geometry]\label{ms:lem:parity-geometry}
Fix $k\ge1$ and use the pointed tree encoding with $n$ edges. Sample $k$ independent uniform tree vertices $v_1,\ldots,v_k$. Let $\mathcal X_n$ be the encoded quadrangulation with distances multiplied by $c_n=(9/(8n))^{1/4}$, marked at these vertices, and equipped with the uniform probability measure on its $n+1$ tree vertices. Put $m_n=\lceil\log n\rceil$.

With probability tending to one, one can select, using only the tree and the sampled vertices, $m_n$ edges on each root-to-$v_i$ path which belong to none of the other sampled root paths. Replace all selected increments by fresh uniform draws from $\{-1,0,1\}$, independent of one another and of the entire original encoding and sampled vertices. Retain the tree, marks and root-orientation bit; if such a selection is impossible, change nothing. The resulting marked space $\widetilde{\mathcal X}_n$ has the same marginal law as $\mathcal X_n$. Their marked Gromov--Hausdorff--Prokhorov distance tends to zero in probability. Both marginal laws converge to the unit-area Brownian map with $k$ conditionally independent volume-sampled marks.
\end{Lemma}
\begin{proof}
Write $Y_n$ for the rescaled contour--label--distance triple in Theorem~\ref{ext:encoding-metric-convergence}. We first deduce that $Y_n$ converges in distribution along the full sequence to $(e,Z,D^*)$. The theorem gives a convergent further subsequence from every subsequence, and identifies its limit as $(e,Z,D^*)$ with $D^*$ determined by the same pair $(e,Z)$. All these limits therefore have the same joint law. To verify full-sequence convergence, let $F$ be any bounded continuous function on $C([0,1])^2\times C([0,1]^2)$, equipped with the uniform norms. If $\E F(Y_n)$ did not converge to $\E F(e,Z,D^*)$, some subsequence would remain at a fixed positive distance from this value. A convergent further subsequence supplied by the theorem would force these expectations to converge to that value, a contradiction. This proves the required joint functional convergence.

List the tree vertices in preorder as $v^{(0)},\ldots,v^{(n)}$. The first contour visit of $v^{(j)}$ occurs at time $2j-\operatorname{depth}(v^{(j)})$. The functional convergence just proved implies that the maximum depth divided by $\sqrt n$ is bounded in probability: for every $\varepsilon>0$ there is a finite $M$ such that the probability that this ratio exceeds $M$ is at most $\varepsilon$ for all sufficiently large $n$. Thus these first-visit times, divided by $2n$, differ uniformly from $j/n$ by a quantity tending to zero in probability. Independent uniform vertex indices consequently give contour times converging jointly to independent uniform times $U_1,\ldots,U_k$ on $[0,1]$, independent of the limiting excursion.

For the reduced tree joining the root and the sampled vertices, the terminal length leading to $U_i$ is
\[
e(U_i)-\max\left(\{0\}\cup
 \left\{\min_{u\in[U_i\wedge U_j,U_i\vee U_j]}e(u):j\ne i\right\}\right).
\]
This is strictly positive almost surely. Indeed, the uniform times are distinct and lie in $(0,1)$ almost surely; near any independent uniform time the excursion has values strictly below its value on both sides. For each deterministic $t\in(0,1)$, the Markov property in \cite[Proposition~2.9, p.~165]{legallmiermont2012} and the absolute-continuity formula immediately following it transfer Brownian oscillation just after $t$ to the normalized excursion: one first works up to any fixed time $b\in(t,1)$. The time-reversal invariance stated on the same page gives the corresponding oscillation just before $t$. Fubini's theorem then gives both properties at an independent uniform time. Each of the finitely many interval minima is therefore strictly smaller than $e(U_i)$. Uniform contour convergence now shows that the discrete terminal lengths, divided by $\sqrt{2n}$, converge to these positive lengths. Since $m_n/\sqrt n\to0$, the required selection succeeds with probability tending to one. For definiteness select the last $m_n$ edges of each private terminal path.

Selection depends only on the tree and sampled vertices. The label increments on the selected edges are independent of that information and of the other increments. Replacing them by independent increments of the same distribution therefore preserves the joint encoding law, and Theorem~\ref{ext:pointed-tree-coding} identifies both encoded maps with the required pointed quadrangulation law. The maximum change in a vertex label is at most $2km_n$, so the rescaled old and new label functions differ uniformly by $o(1)$.

Let $D_n,\widetilde D_n$ be the old and new rescaled distance fields on contour times. Their contour--label--distance triples are jointly tight because each has the full-sequence marginal convergence proved above. In any joint subsequential limit their contour and label functions coincide. Theorem~\ref{ext:encoding-metric-convergence} identifies each limiting distance field as $D^*$ of this same pair $(e,Z)$, hence the two limiting fields coincide. It follows that
\[
\|D_n-\widetilde D_n\|_\infty\longrightarrow0
\quad\text{in probability}.
\]
This uses identification of all subsequential limits, not deterministic continuity of the discrete encoding map.

Correspond each old tree vertex to the same vertex in the new encoding. The distortion on these vertices is at most the displayed uniform difference, the marked vertices are paired exactly, and this correspondence couples their uniform probability measures exactly. Pair also the extra distinguished vertices of the two encodings. By Theorem~\ref{ext:pointed-tree-coding}, their distances to a tree vertex $v$ are given by the corresponding expression $L(v)-\min L+1$, so the rescaled distance discrepancy is at most $4kc_nm_n\to0$. The resulting correspondence proves the claimed marked Gromov--Hausdorff--Prokhorov closeness.

Finally, the empirical first-visit times converge to Lebesgue measure. Combining this with the uniform convergence of the distance fields identifies the limiting vertex measure as Brownian-map volume, and the independent sampled times give independent marks from that measure. The extra distinguished vertex is at rescaled distance $c_n$ from a tree vertex of minimum label, so it creates no additional point in the limiting metric space. This proves the asserted marginal convergence.
\end{proof}
\begin{Lemma}[Absolute label parities and marked geometry]\label{ms:lem:parity-mixing}
Fix an integer $k\ge1$, independently of $n$. In the pointed tree encoding with $n$ edges, sample $k$ independent uniform tree vertices $v_i$ and set $P_i=L(v_i)\pmod2$. Equip the encoded quadrangulation with distances scaled by $(9/(8n))^{1/4}$, its uniform tree-vertex probability measure and these $k$ marks. As $n\to\infty$, the pair consisting of this marked space and $(P_1,\ldots,P_k)$ converges jointly in distribution, for the product of the marked Gromov--Hausdorff--Prokhorov topology and the discrete topology on $\{0,1\}^k$, to the marked Brownian map together with $k$ independent fair bits independent of that space. The limiting $k$ marks are conditionally independent samples from its volume measure.
\end{Lemma}
\begin{proof}
Use the resampling coupling of Lemma~\ref{ms:lem:parity-geometry}, and write $E_n$ for its successful selection event. Condition on the tree, sampled vertices, root-orientation bit, all unselected old increments and all new increments. Call this information $\mathcal F_n$. On $E_n$, each old parity $P_i$ is the parity of a fixed integer plus a sum of $m_n$ independent old increments. These $k$ sums are independent, since the selected paths are private. If $\xi$ is uniform on $\{-1,0,1\}$, then $\E(-1)^\xi=-1/3$. Thus the total variation distance of each such parity from a fair bit is $3^{-m_n}/2$, and the conditional joint distance from the uniform law on $\{0,1\}^k$ is at most $k3^{-m_n}/2$.

The resampled marked space $\widetilde{\mathcal X}_n$ is $\mathcal F_n$-measurable. For every bounded measurable function $f$ of that space and bounded function $g$ on $\{0,1\}^k$, the preceding estimate, with the trivial bound on $E_n^c$, gives
\[
\left|\E[f(\widetilde{\mathcal X}_n)g(P)]
 -\E[f(\widetilde{\mathcal X}_n)]\,2^{-k}\sum_{p\in\{0,1\}^k}g(p)\right|
\le 2\|f\|_\infty\|g\|_\infty
 \left(\Pp(E_n^c)+\frac{k3^{-m_n}}2\right).
\]
The right side tends to zero. For bounded Lipschitz $f$, Lemma~\ref{ms:lem:parity-geometry} permits replacing the resampled space by the original space in this formula: their distance tends to zero in probability, and boundedness gives convergence of the difference in expectation. The marginal convergence from that lemma and these test-function identities prove the claimed joint convergence and independence.
\end{proof}

\subsection{Removing the same-color condition}
We now return from tree labels to uniform vertices of a quadrangulation. The relative colors are intrinsic: reversing the names of the two color classes leaves their differences unchanged.
\begin{Proposition}[The marked limit under same-color conditioning]\label{ms:lem:parity-decoupling}
Fix an integer $k\ge1$. In a uniform rooted planar quadrangulation $Q_n$ with $n$ faces, sample $k$ independent uniform vertices $x_1,\ldots,x_k$. Choose either bipartite coloring $\chi:V(Q_n)\to\{0,1\}$ and put
\[
B_i=\chi(x_i)-\chi(x_1)\pmod2,\qquad 2\le i\le k.
\]
The vector $(B_2,\ldots,B_k)$ is independent of the choice of coloring. Jointly with the marked metric-measure space obtained by multiplying graph distances by $(9/(8n))^{1/4}$ and assigning uniform probability measure to the vertices, it converges in distribution as $n\to\infty$ to $k-1$ independent fair bits independent of the marked Brownian map, whose $k$ marks are conditionally independent samples from its volume measure. Here joint convergence uses the marked Gromov--Hausdorff--Prokhorov topology for compact metric-measure spaces with $k$ ordered marks and the discrete topology on $\{0,1\}^{k-1}$ for the bit vector. Consequently, the probability that all sampled vertices have the same color tends to $2^{1-k}$, and conditioning on that event preserves this marked Brownian-map limit and its Voronoi-volume law.
\end{Proposition}

\begin{proof}
First sample $k$ independent uniform tree vertices in the pointed encoding and write $P_i=L(v_i)\pmod2$. By Lemma~\ref{ms:lem:parity-mixing}, these bits converge jointly with the marked space to independent fair bits independent of that space. The map
\[
(p_1,\ldots,p_k)\longmapsto(p_2-p_1,\ldots,p_k-p_1)\pmod2
\]
sends the uniform law on $\{0,1\}^k$ to the uniform law on $\{0,1\}^{k-1}$, since every relative vector has two preimages. Label differences modulo two are bipartite color differences, by the distance formula in Theorem~\ref{ext:pointed-tree-coding}.

Let $\mathcal X_n$ denote this tree-sampled marked space and $\mathcal X$ its marked Brownian-map limit. If $A_n$ is the event that all relative bits vanish, the discreteness of their state space gives, for every bounded continuous $f$,
\[
 \E[f(\mathcal X_n)\ind_{A_n}]
 \longrightarrow 2^{1-k}\E[f(\mathcal X)],
 \qquad \Pp(A_n)\longrightarrow2^{1-k}>0.
\]
Thus division proves the conditional marked convergence. For $k=1$, the relative bit vector is empty and the event is the whole sample space, so the assertion also holds.

To pass to uniform quadrangulation vertices, add an independent uniform distinguished vertex to a uniform rooted quadrangulation and use Theorem~\ref{ext:pointed-tree-coding}. Its tree vertices are precisely all vertices except the distinguished one. A uniform quadrangulation vertex can be coupled to a uniform tree vertex, with failure probability $1/(n+2)$; for $k$ independent samples the failure probability is at most $k/(n+2)$. The uniform tree and quadrangulation measures differ in total variation by $1/(n+2)$. These errors vanish, including after conditioning since the conditioning probability tends to $2^{1-k}$. This proves both marked-space assertions for quadrangulation samples.

Finally, Lemma~\ref{ms:lem:bisectors} gives null limiting bisectors. Lemma~\ref{ms:lem:cell-mass-continuity}, applied to the unconditional and conditional marked convergence, proves the corresponding convergence of strict Voronoi masses. Any assignment of discrete distance ties has the same limit, since their total mass tends to zero.
\end{proof}

\section{From labelled face areas to Voronoi volumes}\label{sec:faces}
\subsection{The face series and its two interpretations}
The bridge between enumeration and geometry is a generating series with one variable per face. We state the analytic limit as an external theorem, then explain why the same discrete objects encode quadrangulations with ordered same-color sources.

A rooted map has one distinguished oriented edge. A quadrangulation has face degree four, counted with multiplicity, at every face. Its graph distance is $d_Q$ and its vertices carry uniform probability measure. Every planar quadrangulation is bipartite. For $n$ faces it has $n+2$ vertices and $2n$ edges. We use the distance scale $c_n=(9/(8n))^{1/4}$, consistently with Section~\ref{sec:parity}.

We use the face series $\mathcal Z_k$ defined in \eqref{eq:face-series-definition}, with minimum boundary label one on every numbered face and inverse-automorphism weights. An edge whose two sides have face labels $i,j$ contributes $\sqrt{g_i g_j}$; an edge with face $i$ on both sides contributes $g_i$. These are precisely the weights $\prod_i g_i^{a_i}$, where $a_i$ is half the boundary degree and $\sum_i a_i$ is the number of edges. When a formal single-variable interpretation is useful, set $g_i=z w_i$ with fixed positive $w_i$ and expand in $z$. For positive numerical weights we retain the nonnegative-sum interpretation from Section~\ref{sec:introduction}.

Throughout this section the critical scaling is
\begin{equation}\label{eq:critical-scaling}
 g_i(\varepsilon)=\frac{1-\lambda_i\varepsilon^4/36}{12},\qquad
 \eta_i=\lambda_i/4,
\end{equation}
with positive $\lambda$ fixed and $\varepsilon\downarrow0$.

The analytic input is \cite[Theorem~1.1, p.~5]{sun2026critical}, with the series defined on its page~3 and the kernels and marked scheme integral on pages~4--5. Its definitions agree with those above: the same inverse-automorphism weights, half-face degrees, Friedrichs Green kernels with derivative jump $-1$, and all four face-mark cases are used, including the factor two when an edge has the same face on both sides and both orders of two marks.
\begin{ExternalTheorem}[Critical limit of the face-constrained series]\label{ext:face-constrained-critical-limit}
For every fixed $k\ge3$, the integral $\mathcal I_{\{1,\ldots,k\}}$ is finite and continuous on $(0,\infty)^k$, and
\[
 \varepsilon^{4k-10}\mathcal Z_k(g_1(\varepsilon),\ldots,g_k(\varepsilon))
 \longrightarrow 3^{3k-6}\mathcal I_{\{1,\ldots,k\}}(\lambda/4)
\]
locally uniformly for $\lambda$ in compact subsets of $(0,\infty)^k$.
\end{ExternalTheorem}
Only this finite, continuous critical limit is imported. Real analyticity and the diagonal formula will follow from the probabilistic identification below.

We use the zero-delay specialization of Miermont's correspondence \cite[Sections 2.2 and 2.5, Lemma 2 and Theorem 4, pp. 733--734 and 738--740]{miermont2009}. His labels are defined modulo a common additive constant; requiring every face minimum to be one fixes the representative for which every source has label zero.
\begin{ExternalTheorem}[Zero-delay correspondence]\label{ext:general-zero-delay}\label{ext:zero-delay-correspondence}
For each $k\ge3$, there is a bijection between the maps counted by $\mathcal Z_k$ with $n$ edges and unrooted planar quadrangulations with $n$ faces and $k$ ordered distinct vertices $x_1,\ldots,x_k$ in one bipartite color class, both taken up to orientation-preserving isomorphism. It matches source $x_i$ with numbered face $i$ and identifies the automorphism group preserving the ordered sources with the automorphism group preserving the face numbers and vertex labels. Thus corresponding objects have the same inverse-automorphism weight, with no multiplicity factor. The vertices other than the sources are the vertices of the labelled map, with
\[
 \ell(v)=\min_i d_Q(v,x_i),\qquad
 d_Q(v,x_i)=\ell(v)\quad\text{if $v$ is incident to face $i$}.
\]
Each labelled-map edge corresponds to one quadrangulation face.
\end{ExternalTheorem}
The group identification uses the same constructions: the forward construction and the inverse successor rule depend only on the cyclic orders, labels and numbered faces or ordered sources, so both commute with orientation-preserving isomorphisms.
The delay conditions require $|\tau_i-\tau_j|<d_Q(x_i,x_j)$ for $i\ne j$ and $d_Q(x_i,x_j)+\tau_i-\tau_j$ to be even. With all delays $\tau_i=0$, distinct sources in the same bipartite class satisfy
\[
 |\tau_i-\tau_j|=0<d_Q(x_i,x_j),\qquad d_Q(x_i,x_j)\in2\mathbb Z,
\]
so both requirements hold for every fixed $k$.

We use the fixed-face pruning and chain-suppression decomposition of Miermont~\cite[Section~3.1]{miermont2009}; Bouttier and Guitter~\cite[Section~3.2]{bouttierguitter} describe the three-face case. Prune edges incident to vertices of degree one, then suppress every maximal degree-two chain. The resulting reduced scheme has minimum degree three and retains the ribbon structure and face numbers. The pruned plane trees are attached in corners, the sectors between consecutive incidences at a vertex. Their root labels are the labels of their attachment vertices. A chain is nonempty, even when its endpoint labels coincide; a loop retains both endpoint incidences. Euler's relation bounds the numbers of vertices and edges of a reduced scheme by $2k-4$ and $3k-6$, respectively, so there are finitely many schemes for fixed $k$.

\subsection{Fixed-size face areas}
Write $M_n$ for a labelled map with $n$ edges and $Q_n$ for its corresponding quadrangulation under Theorem~\ref{ext:general-zero-delay}. The following dictionary separates the two objects:
\begin{center}
\begin{tabular}{@{}ll@{}}
\toprule
Labelled map $M_n$ & Quadrangulation $Q_n$ \\
\midrule
$n$ edges & $n$ quadrangular faces \\
$k$ numbered faces & $k$ ordered sources $x_1,\ldots,x_k$ \\
Half-face degree $a_i=\deg_{M_n}(i)/2$ & Discrete area associated with $x_i$ \\
Label $\ell(v)$ & Nearest-source distance $\min_j d_{Q_n}(v,x_j)$ \\
\bottomrule
\end{tabular}
\end{center}
The varying face areas are the half-degrees in $M_n$; every face of $Q_n$ still has degree four. At fixed $n$, we sample $M_n$ with probability proportional to its inverse automorphism weight. The face fugacities will be inserted later as observables under this law.

There are two assertions to prove: a deterministic comparison of $a_i$ with strict Voronoi vertex counts in $Q_n$, up to a tie error, and convergence of those counts under the correct conditioned law. We establish the law first so that both the geometric convergence and the tie estimate apply to the object being counted.
\begin{Lemma}[From face areas to Voronoi volumes]\label{ms:lem:v8-face-area-limit}
Fix an integer $k\ge3$. For $n\ge k-1$, let $\mathfrak M_{n,k}$ be the set of isomorphism classes of maps counted by $\mathcal Z_k$ with $n$ edges, and sample $M_n\in\mathfrak M_{n,k}$ with law
\[
 \Pp(M_n=M)=\frac{|\Aut(M)|^{-1}}
 {\displaystyle\sum_{M'\in\mathfrak M_{n,k}}|\Aut(M')|^{-1}},
 \qquad M\in\mathfrak M_{n,k}.
\]
No face fugacity factors enter this fixed-size probability law. Writing $a_i=a_i(M_n)$ for its half-face degrees, as $n\to\infty$ we have
\[
(a_1/n,\ldots,a_k/n)\Longrightarrow(V_1^{(k)},\ldots,V_k^{(k)}).
\]
The associated quadrangulations have $k$ ordered distinct same-color sources. The number of vertices with more than one nearest source is $o(n)$ in probability, and changing from face areas to strict Voronoi vertex counts changes each normalized coordinate by $o(1)$.
\end{Lemma}
\begin{proof}
Start with a uniform rooted quadrangulation with $n$ faces, sample $k$ independent uniform vertices, and condition on their being distinct and of the same color. Before conditioning, every rooted quadrangulation with an ordered vertex tuple has the same probability, since every quadrangulation has $n+2$ vertices. The conditioned law is consequently uniform on the admissible rooted marked objects. For an unrooted quadrangulation with ordered sources, the action of its source-preserving automorphism group on the $4n$ oriented roots is free. Forgetting the root therefore gives weight proportional to $4n/|\Aut(Q_n;x_1,\ldots,x_k)|$. Theorem~\ref{ext:general-zero-delay} identifies this group with $\Aut(M_n)$, so the induced law of $M_n$ is exactly the inverse-automorphism law in the statement.

The probability of a collision among the sampled vertices is at most $\binom{k}{2}/(n+2)$. By Proposition~\ref{ms:lem:parity-decoupling}, the same-color event has probability tending to $2^{1-k}>0$, and conditioning on it preserves the marked Brownian-map limit. Further conditioning on distinctness changes that conditional law by $O_k(n^{-1})$ in total variation. Thus $Q_n$, with distances scaled by $c_n=(9/(8n))^{1/4}$, uniform vertex measure and its $k$ ordered sources, converges to the Brownian map with $k$ independent volume-sampled marks.

Let $W_{n,i}$ be the number of vertices in the strict Voronoi cell of $x_i$ in $Q_n$, and let $T_n$ be the number of vertices with more than one nearest source. Null limiting bisectors (Lemma~\ref{ms:lem:bisectors}) and continuity of cell masses (Lemma~\ref{ms:lem:cell-mass-continuity}) now give
\begin{equation}\label{eq:fixed-size-vertex-limit}
 \left(\frac{W_{n,1}}{n+2},\ldots,\frac{W_{n,k}}{n+2}\right)
 \Longrightarrow (V_1^{(k)},\ldots,V_k^{(k)}),\qquad
 \frac{T_n}{n+2}\longrightarrow0\quad\text{in probability}.
\end{equation}
This is the probabilistic part of the argument. The remaining comparison with $a_i$ is deterministic.

Let $N_{n,i}$ count vertices of $M_n$ incident only to face $i$. In the pruning decomposition, each removed tree lies inside one face and contributes equally many edges and new vertices. Since each such edge is counted twice in that face's boundary, its contribution to the half-degree equals its contribution to the exclusive vertex count. After pruning, a degree-two vertex incident only to face $i$ contributes one to both counts. A degree-two vertex incident to two faces contributes at most one to their discrepancy for a fixed face. The other vertices belong to the reduced scheme, whose total degree is at most $6k-12$. It follows that, for a constant $D_k$ depending only on $k$,
\[
 |a_i-N_{n,i}|\le J_n+D_k,
\]
where $J_n$ counts vertices of $M_n$ incident to more than one face.

By Theorem~\ref{ext:general-zero-delay}, every incident face corresponds to a minimizing source. Hence every vertex counted by $J_n$ is a nearest-source tie, so $J_n\le T_n$. Conversely, a nonsource vertex with a unique nearest source $x_i$ is incident only to face $i$. Thus any difference between $N_{n,i}$ and the strict cell count over nonsource vertices is again at most $T_n$. The labelled map omits exactly the $k$ sources; each distinct source belongs to its own strict cell. Restoring $x_i$ adds one to that cell count, and therefore
\[
 |a_i-W_{n,i}|\le 2T_n+D_k+1.
\]
Since $W_{n,i}\le n+2$, changing the normalization from $n$ to $n+2$ gives
\begin{equation}\label{eq:face-vertex-comparison}
 \left|\frac{a_i}{n}-\frac{W_{n,i}}{n+2}\right|
 \le\frac{2T_n+D_k+3}{n}\longrightarrow0
 \quad\text{in probability}.
\end{equation}
Combining \eqref{eq:face-vertex-comparison} with \eqref{eq:fixed-size-vertex-limit} proves the face-area limit.

For clarity, none of the preceding probability laws includes unequal face weights. With $g_i=z w_i$, the series at size $n$ inserts the observable $\prod_i w_i^{a_i}$ under the inverse-automorphism law. Dividing that observable by its expectation would define a different, tilted law; no convergence under such a tilt is needed here.
\end{proof}

\subsection{The counting constant}
The correspondence preserves inverse automorphism weights. It therefore suffices to count rooted quadrangulations with ordered sources, impose the same-color condition, and then forget the root. Each of these operations contributes to the constant.
\begin{Lemma}[The source and rooting constant]\label{ms:lem:v8-counting-normalization}
Let $A_{n,k}$ be the sum of inverse automorphism weights of maps counted by $\mathcal Z_k$ with $n$ edges. For each fixed $k\ge3$, there exists a finite constant $C_k'>0$, depending only on $k$, such that
\[
A_{n,k}\sim\frac{2^{-k}}{\sqrt\pi}12^n n^{k-7/2},\qquad
A_{n,k}\le C_k'12^n n^{k-7/2}.
\]
The asymptotic is as $n\to\infty$, and the inequality holds for every integer $n\ge1$.
\end{Lemma}
\begin{proof}
Let $q_n$ be the number of rooted planar quadrangulations with $n$ faces. The bijection in Theorem~\ref{ext:pointed-tree-coding} identifies rooted quadrangulations with a distinguished vertex with a plane tree with $n$ edges, $n$ label increments in $\{-1,0,1\}$, and a root-orientation bit. There are $\frac1{n+1}\binom{2n}{n}$ such trees and $2\,3^n$ choices of increments and bit for each tree. A connected rooted map has no nontrivial root-preserving automorphism, so its $n+2$ vertices give $n+2$ distinct pointed rooted maps. Consequently
\[
 (n+2)q_n=2\,3^n\frac1{n+1}\binom{2n}{n},\qquad
 q_n\sim\frac2{\sqrt\pi}12^n n^{-5/2}.
\]
Let $p_{n,k}$ be the probability that $k$ independent uniform vertices of a uniform rooted quadrangulation are distinct and of the same color. Proposition~\ref{ms:lem:parity-decoupling} and the collision bound give $p_{n,k}\to2^{1-k}$. Counting all $(n+2)^k$ ordered tuples, retaining the admissible ones and forgetting the root yields the exact relation
\[
 A_{n,k}=\frac{q_n(n+2)^k p_{n,k}}{4n}.
\]
Indeed, the free automorphism action on the $4n$ oriented roots gives $4n/|\Aut|$ rooted objects for each unrooted marked isomorphism class, and the correspondence preserves that automorphism group. This accounts separately for the source choices, the color condition and the root. Substituting the asymptotics gives
\[
 A_{n,k}\sim\frac{2^{-k}}{\sqrt\pi}12^n n^{k-7/2},
 \qquad A_{n,k}\le C_k'12^n n^{k-7/2}.
\]
The bound follows from the exact formula for $q_n$ and at most $(n+2)^k$ source choices; finitely many small $n$ are absorbed into $C_k'$.
\end{proof}

\section{The critical comparison}\label{sec:comparison}
The fixed-size result determines the limiting Laplace transform
\[
 \varphi(t)=\E\exp\left(-t\sum_{i=1}^k\lambda_iV_i^{(k)}\right),\qquad t\ge0.
\]
It does not yet identify the analytic limit in Theorem~\ref{ext:face-constrained-critical-limit}, because that theorem concerns the generating series summed over all sizes. Put $h=\varepsilon^4/36$ and $\alpha_k=k-5/2>0$. A size $n$ contributes on the scale $t=nh$, and Lemma~\ref{ms:lem:v8-counting-normalization} gives
\[
 12^{-n}A_{n,k}\sim\frac{2^{-k}}{\sqrt\pi}n^{\alpha_k-1}.
\]
Thus the size sum becomes an integral with weight $t^{\alpha_k-1}$. The gamma identity then turns the Laplace transform into a negative-power transform:
\[
 \int_0^\infty t^{\alpha_k-1}\varphi(t)\,dt
 =\Gamma(\alpha_k)\E\left(\sum_{i=1}^k\lambda_iV_i^{(k)}\right)^{-\alpha_k}.
\]
The next proposition justifies this passage, including the contributions from sizes much smaller or much larger than $h^{-1}$.
\begin{Proposition}[Probabilistic normalization of the face series]\label{ms:prop:general-normalization}
For every fixed $k\ge3$ and positive rate vector $\lambda$,
\[
 \varepsilon^{4k-10}\mathcal Z_k(g_1(\varepsilon),\ldots,g_k(\varepsilon))
 \longrightarrow C_k\E\left(\sum_{i=1}^k\lambda_iV_i^{(k)}\right)^{-\alpha_k}.
\]
\end{Proposition}
\begin{proof}
We first express the face series at fixed size using inverse-automorphism weights, then sum the limiting Laplace transforms over the size parameter. The geometric input is Lemma~\ref{ms:lem:v8-face-area-limit}; the counting input is Lemma~\ref{ms:lem:v8-counting-normalization}. Write $A_{n,k}$ for their total. For sizes with $A_{n,k}>0$, let $\E_{n,k}$ denote expectation under the probability assigning a map with $n$ edges mass $1/(A_{n,k}|\Aut|)$. Terms with $A_{n,k}=0$ in the sum below are zero. Lemma~\ref{ms:lem:v8-face-area-limit} gives convergence of the normalized face areas under this law. The parameters $g_i$ do not enter this probability measure: their entire contribution remains inside the expectation below. Lemma~\ref{ms:lem:v8-counting-normalization} also gives the required asymptotic for $A_{n,k}$ and its uniform bound. Since $\sum_i a_i=n$,
\[
 \mathcal Z_k(g(\varepsilon))=\sum_{n\ge1}12^{-n}A_{n,k}
 \E_{n,k}\prod_i(1-\lambda_i\varepsilon^4/36)^{a_i}.
\]
Write $h=\varepsilon^4/36$ and $t=nh$. If $nh$ remains in a compact subset of $(0,\infty)$, Taylor expansion gives, uniformly in the possible area vectors,
\[
 \sum_i a_i\log(1-\lambda_i h)
 =-nh\sum_i\lambda_i(a_i/n)+O(nh^2).
\]
The error tends to zero uniformly there. The face-area convergence therefore gives the limiting factor
$\varphi(t)=\E\exp(-t\sum_i\lambda_iV_i^{(k)})$.
This convergence is uniform on each compact $t$ interval: the functions
$\E_{n,k}\exp(-t\sum_i\lambda_i a_i/n)$ have Lipschitz constants at most $\max_i\lambda_i$, so pointwise convergence and a finite mesh imply uniform convergence. In particular it applies along every sequence $n\to\infty$, $nh\to t>0$.

Here are explicit bounds for the parts of the sum outside such an interval. Set $c=\min_i\lambda_i>0$. For $h$ small enough that every $\lambda_i h<1$, positivity and $\sum_i a_i=n$ give
\[
 0\le\prod_i(1-\lambda_i h)^{a_i}\le e^{-cnh}.
\]
Since $4k-10=4\alpha_k$, the coefficient bound in Lemma~\ref{ms:lem:v8-counting-normalization} yields
\[
 \varepsilon^{4\alpha_k}\sum_{nh<\delta}12^{-n}A_{n,k}
 \le C h^{\alpha_k}\sum_{n<\delta/h}n^{\alpha_k-1}
 \le C'\delta^{\alpha_k}.
\]
Thus the small-size contribution vanishes with $\delta$, precisely because $\alpha_k>0$. The large-size contribution is bounded by a constant times
$h\sum_{nh>R}(nh)^{\alpha_k-1}e^{-cnh}$, which tends uniformly to zero as $R\to\infty$ (restrict, for example, to $0<h\le1$ and compare with the integral of an exponential envelope). On $\delta\le nh\le R$, the counting asymptotic and the preceding uniform convergence give ordinary Riemann sums. Letting first $h\downarrow0$, then $\delta\downarrow0$ and $R\uparrow\infty$, proves that the scaled series tends to
\[
 \frac{2^{-k}36^{\alpha_k}}{\sqrt\pi}
 \int_0^\infty t^{\alpha_k-1}
 \E e^{-t\sum_i\lambda_iV_i^{(k)}}\,dt
 =C_k\E\left(\sum_i\lambda_iV_i^{(k)}\right)^{-\alpha_k}.
\]
Tonelli's theorem and the gamma integral justify the last equality.
\end{proof}

\begin{proof}[Proof of Theorem~\ref{ms:lem:general-bridge}]
Theorem~\ref{ext:face-constrained-critical-limit} and Proposition~\ref{ms:prop:general-normalization} compute the limit of the same rescaled generating series, with the same inverse-automorphism weights and critical parameters. The external theorem gives the finiteness of the integral. Comparing the two limits and relabelling $\{1,\ldots,k\}$ by $S$ proves \eqref{eq:bridge} for every finite label set and every positive rate vector.

For completeness, we give the analytic dependence explicitly. By Lemma~\ref{ms:lem:bisectors}, the nonnegative volumes sum to one almost surely. Fix a positive vector $\lambda^0$ and set $b=\min_i\lambda_i^0>0$ and $A=\sum_i\lambda_i^0V_i$. For $\|h\|_\infty<b/2$, the weighted perturbation $B=\sum_i h_iV_i$ satisfies $|B|/A<1/2$ and $A\ge b$. Hence
\[
 (A+B)^{-\alpha_k}
 =A^{-\alpha_k}\sum_{m=0}^\infty
   \frac{(-1)^m(\alpha_k)_m}{m!}\left(\frac BA\right)^m,
\]
where $(\alpha)_0=1$ and $(\alpha)_m=\alpha(\alpha+1)\cdots(\alpha+m-1)$ for $m\ge1$. The series converges absolutely and uniformly in the volume vector on each smaller neighborhood of $h=0$. Each term is a homogeneous polynomial in $h$, with integrable coefficients. Taking expectations gives a locally convergent power series for $F_S$, and thus proves real analyticity. Equation~\eqref{eq:bridge} transfers this property to the integral. The lower bound $\sum_i\lambda_iV_i\ge\min_i\lambda_i$ also proves directly that $F_S(\lambda)$ is finite.

On the diagonal $\lambda_i=4t$, the weighted sum of the volumes is $4t$ almost surely. Therefore
\[
 \mathcal I_S((t)_{i\in S})
 =3^{-(3k-6)}C_k(4t)^{-(k-5/2)}
 =\frac{\Gamma(k-5/2)}{2^k3^{k-1}\sqrt\pi}\,t^{5/2-k},
\]
which is \eqref{eq:diagonal}. This holds for every $t>0$.
\end{proof}

The same calculation and homogeneity give the equivalent normalization
\[
 \mathcal I_S(\eta)=
 \frac{\Gamma(k-5/2)}{2^k3^{k-1}\sqrt\pi}\,F_S(\eta),
 \qquad \eta\in(0,\infty)^S.
\]
Thus evaluating the integral for unequal parameters would evaluate the corresponding volume transform. No such evaluation or identification of the joint volume distribution is needed for the representation proved here.

\section*{Acknowledgements}

\begingroup
\emergencystretch=2em
The author acknowledges support from the National Key R\&D Program of China (Grant No.~2023YFA1010102) and the National Natural Science Foundation of China (Grant Nos.~12301173 and 12671171).
\par\endgroup

The author acknowledges the use of large language models (OpenAI's GPT-6/6.1), through an iterative prompting process, to support the conceptual development and proof construction in this paper. The author has independently verified all mathematical claims and arguments, and assumes full responsibility for the correctness and academic integrity of the final results.

\bibliographystyle{amsplain}
\bibliography{references}
\end{document}